\documentclass[11pt]{article}
\usepackage{amssymb,amsmath}
\usepackage{amsthm,color}
\newtheorem{theorem}{Theorem}[section]
\newtheorem{theorem*}{Theorem A\!\!}
\newtheorem{proposition}{Proposition}[section]
\newtheorem{proposition*}{Proposition A\!\!}
\newtheorem{corollary}{Corollary}[section]
\newtheorem{corollary*}{Corollary A\!\!}

\newtheorem{lemma}{Lemma}[section]

\newtheorem{definition}{Definition}[section]

\DeclareMathOperator{\Id}{Id}

\begin{document}

\title{Singular conformally invariant trilinear forms and generalized Rankin Cohen operators}

\author{Ralf Beckmann \qquad Jean-Louis Clerc}
\date{April 18, 2011}
\maketitle

\hskip6cm {\it Dedicated to Elias Stein}

\hskip 6cm {\it on the occasion of his 80th birthday}

\begin{abstract} The most singular residues of the standard meromorphic family of trilinear conformally invariant forms on $\mathcal C^\infty_c(\mathbb R^d)$ are computed. Their expression involves covariant bidifferential operators (generalized Rankin Cohen operators), for which new formul\ae \ are obtained. The main tool is a Bernstein-Sato identity for the kernel of the forms.
\end{abstract}

\footnotemark[0]{2000 Mathematics Subject Classification : 22E45, 43A85}

\section*{Introduction}

Let $E$ be a finite dimensional Euclidean space of dimension $d, (d\geq 2)$\footnote{The case $d=1$ could be treated along the same lines, but there are some differences, which would require separate statements. See \cite{mol} for a study of this case.}. Let $G\simeq SO_0(1,d+1)$ be the (connected component of) the conformal group of $E$, acting  by rational transformations on $E$. Much interest has been devoted to various invariant or covariant objects for this action. There is a natural action of $G$ on a space of densities on $E$. Identifying the densities with functions on $E$, the action is given by
\[\pi_\lambda(g) f(x) = \kappa(g^{-1}, x)^{\rho+\lambda}f(g^{-1}(x))\ ,
\]
where $\kappa(g,x)$ is the \emph{conformal factor} (the infinitesimal rate of dilation) of the transformation $g\in G$ at $x\in E$, $\lambda$ is a complex parameter and $\rho=\frac{d}{2}$. Let $\lambda_1,\lambda_2,\lambda_3$ be three complex numbers. A continuous  trilinear form $\mathcal L$ on $\mathcal C^\infty_c(E)\times \mathcal C^\infty_c(E) \times \mathcal C^\infty_c(E)$ is said to be \emph{conformally invariant} with respect to $\pi_{\lambda_1}\otimes \pi_{\lambda_2}\otimes \pi_{\lambda_3}$, if, for any three functions $f_1,f_2,f_3\in \mathcal C^\infty_c(E)$
\[\mathcal L(\pi_{\lambda_1}(g) f_1, \pi_{\lambda_2}(g) f_2,\pi_{\lambda_3}(g) f_3)= \mathcal L(f_1,f_2,f_3)
\]
where $g$ is in $G$ (strictly speaking, defined on the union of the supports of the three functions). These trilinear forms have been investigated in a previous work of the second author in collaboration with B. \O rsted (see \cite{co}). Generically, for $\boldsymbol \lambda =(\lambda_1,\lambda_2,\lambda_3)$ in $\mathbb C^3$, there is a unique (up to a multiple) such invariant trilinear form. Viewing the trilinear form as a distribution on $E\times E\times E$, it has a smooth density on the open set \[\{(x_1,x_2,x_3)\in E\times E\times E \,; x_1\neq x_2, x_2\neq x_3, x_3\neq x_1\}\ , \] given by
\[l_{\boldsymbol \beta}(x_1,x_2,x_3) = \vert x_1-x_2\vert^{\beta_3}\vert x_2-x_3\vert^{\beta_1} \vert x_3-x_1\vert^{\beta_2}
\]
where $\boldsymbol \beta= (\beta_1,\beta_2,\beta_3)$ is a triplet of complex numbers, uniquely determined by $\boldsymbol \lambda =(\lambda_1,\lambda_2,\lambda_3)$ (see \eqref{betalambda}). The corresponding distribution $\mathcal L_{\boldsymbol \beta}$ is defined by meromorphic continuation, and has simple poles along certain planes in $\mathbb C^3$. The study of the   residues was begun by the second author in \cite{c}, and the present paper deals with the most singular residues. They are distributions supported on the diagonal $\mathcal D = \{(x,x,x)\, ; x\in E\}$. They turn out to be intimately related with \emph{covariant bidifferential operators}, that is differential operators from $\mathcal C^\infty_c(E\times E)$ into $\mathcal C^\infty_c(E)$ which satisfy a relation of the form
\[ D(\pi_\lambda(g) f_1\otimes \pi_{\mu}(g)f_2) = \pi_\nu(g) D(f_1\otimes f_2)
\]
for some $(\lambda,\mu,\nu)\in \mathbb C^3$ and  for any $g\in G, f_1,f_2\in \mathcal C^\infty_c(E)$. Such operators have a long history (see \cite{d}, \cite{e}), and the most celebrated ones are the \emph{Rankin Cohen operators}, which are holomorphic bidifferential operators on the complex upper half plane and covariant under the group $PSL_2(\mathbb R)$. In the context of conformal geometry, they were studied by Ovsienko and Redou (\cite{or}). See also \cite {kr}, \cite{m}.

The basic ingredient we use for computing the residues is a \emph{Bernstein-Sato identity} (cf Theorem \ref{bs}). Recall that for $f_1,f_2,f_3$ three arbitrary nonnegative polynomials on $E$, there exists a differential operator $B= B(x_1,x_2,x_3,s_1,s_2,s_3,\partial_{x_1}, \partial_{x_2}, \partial_{x_3})$ on $E$ which is polynomial in the $x_j$ and $s_j$, and a polynomial $b$ on $\mathbb C^3$, such that
\begin{equation}
B (f_1^{s_1+1} f_2^{s_2} f_3^{s_3}) = b(s_1,s_2,s_3) f_1^{s_1} f_2^{s_2}f_3^{s_3}\ .
\end{equation}
Such identities exist in general (cf \cite{sab1}, \cite{sab2}), but their explicit determination is seldom known. Once such an identity is known,  the computation of the residues is easy. As a consequence, we find new expressions  for the covariant bidifferential operators alluded to previously. More general results on Bernstein-Sato identities will appear in \cite{b}.

The plan of the paper is as follows. Section 1 collects some results on distributions supported on a subspace. Section 2 is an elementary approach to computing the residues, much in the spirit of  Gelfand and Shilov (\cite {gs}). In particular it allows to determine the residues along the "first" plane of poles. Section 3 exploits  the invariance property of the trilinear forms and the covariance property of the associated bidifferential operators, making connection with the results of \cite{or}. Section 4 is devoted to the Bernstein-Sato identity (Theorem \ref{bs}), which is used in section 5 to give a formula for the residues, and as a consequence, a new expression for the covariant bidifferential operators (see  \eqref{covbidiff2}). Section 6 presents several remarks and perspectives on the subject.

The second author wishes to dedicate the present paper to Elias Stein, on the occasion of his 80th birthday.

\section{Distributions supported on a subspace and transverse differential operators}

Let $E$ be a finite dimensional real vector space, and let $V$ be a linear subspace of $E$. Let $E'$ be the dual space of $E$, and let
\[V^\perp = \{ \xi\in E'\ ;\ \xi_{\vert V} = 0\}\ .
\]
Let $u$ be a distribution on $V$. The assignment
\[\mathcal C^\infty_c(E)\ni\varphi \longmapsto (u, \varphi_{\vert V})
\]
defines a distribution on $E$, the natural extension of $u$, hereafter denoted by $\widetilde u$. Clearly $Supp(\widetilde u) = Supp(u)\subset V$. We now characterize the \emph{wavefront set} of $\widetilde u$ (cf \cite {h}).
\begin{proposition}\label{pwf}
 Let $u$ be in $\mathcal D'(V)$, and let $\widetilde u$ the associated distribution on $E$. Then
\begin{equation}\label{wf}
WF(\widetilde u) = \{ (x,\xi)\in Supp(u)\times (E'\setminus 0)\,;\,\xi \in V^\perp {\rm \ or\ }(x,\xi_{\vert V})\in WF(u) \}\ .
\end{equation}
\end{proposition}
\begin{proof} Choose a subspace $W$ such that $E=V\oplus W$. For $\xi\in E'$,  let $\xi=\xi'+\xi''$, where $\xi'\in W^\perp$ and $\xi''\in V^\perp$.
Let $\varphi$ be in $\mathcal C^\infty_c(E)$. 
Then $\varphi \widetilde u$ is a distribution with compact support and its Fourier transform is given by
\begin{equation*}
\mathcal F(\varphi \widetilde u)(\xi)=(u, e^{-i(\xi',.)} \varphi_{\vert V})\ .
\end{equation*}
Let $(x_0,\xi_0)$ be in the set described by the RHS of \eqref{wf}. If $\xi'_0=0$, then choose $\varphi$ such that  $\langle u ,\varphi_{\vert V}\rangle \neq 0$ (which is always possible since $x_0$ belongs to  $Supp(u)$), so that $\mathcal F(\varphi \widetilde u)(\xi)$  cannot decrease rapidly in a conic neighbourhoud of $\xi_0$. If $(x_0,\xi'_0)$ belongs to $WF(u)$, $\mathcal F(\varphi \widetilde u)(\xi)$ cannot decrease rapidly on a conic neigbourhood of  $\xi'_0$ in $W^\perp\setminus 0$, {\it a fortiori}  on a conic neigbourhood of $\xi_0$ in $E'\setminus 0$. Conversely, assume $(x_0,\xi_0')$ does not belong to $WF(u)$ and $\xi'_0\neq0$. Then, for $\varphi$ with a sufficently small support near $x_0$ and for $\xi$ in a (small enough) conic neigbourhood of $\xi_0$, $\mathcal F(\varphi \widetilde u)(\xi)$ can be dominated by 
$C_N(1+\vert \xi'\vert)^{-N}$ for any integer $N$. But in a (sufficently small) conic neighbourhood of $\xi'_0$ one has $\vert \xi\vert \leq C \vert\xi'\vert$ for some constant $C>0$, so that $\mathcal F(\varphi \widetilde u)(\xi)$ is dominated by $C_N (1+\vert \xi\vert)^{-N}$. Hence $(x_0,\xi_0)\notin WF(\widetilde u)$.
\end{proof}

To further investigate distributions supported on $V$, one needs to introduce \emph{normal derivatives}. Fix a splitting $E=V\oplus W$ as above, and choose ccordinates $w_1,w_2,\dots, w_p$ on $W$, which can be regarded as (a partial set of) coordinates on $E$ by extending them by $0$ on $V$. Let $I=(i_1,\dots, i_p)$ be a p-tuple of natural integers, let $\vert I\vert = i_1+i_2 +\dots +i_p$. Let $D_I$ be the operator (the $D_I$'s are often referred to as \emph{normal derivatives}), defined by
\[D_I\varphi\,(v) = \frac{\partial^{\vert I\vert}\varphi}{\partial w_1^{i_1}\dots  \partial w_p^{i_p}}(v),
\]
mapping smooth functions on $E$ to smooth functions on $V$. To any distribution $u$ on $V$, one can associate the distribution $D_I\widetilde u$ defined by
\[ (-1)^{\vert I\vert}( D_I \widetilde u , \varphi)  =(u, D_I \varphi)\ .
\]
Observe that $WF(D_I\widetilde u) = WF(\widetilde u)$. The inclusion $\subset$  is obvious, whereas the opposite inclusion is obtained by testing against functions $\varphi$ of the form 
\begin{equation}\label{tf}
\varphi(v,w) = \chi(v) w^I \psi(w),
\end{equation}
 where $\chi \in \mathcal C_c^\infty(V)$, $w^I = w_1^{i_1}\dots w_d^{i_d}$ and $\psi$ is a function in $\mathcal C_c^\infty(W)$ which is identically equal to $1$ in a neigbourhood of $0$.

Now let $U$ be a distribution on $E$, with $Supp(U)\subset V$. The structure theorem of L. Schwartz asserts that there exist distributions $u_I$ on $V$ such that
\[U= \sum_{I} D_I \widetilde u_I,
\]
where the sum is locally finite. Moreover, the $u_I$'s are unique.

If all the distributions $u_I$ are given by smooth densities, then from  \eqref{wf}, $WF(U) \subset  E\times (V^\perp\setminus 0)$. The converse is true.

\begin{proposition}\label{smooth}
 Let $U$ be a distribution supported in $V$, and assume that
\[WF(U)\subset V\times (V^\perp \setminus 0)\ .
\]
Then there exist smooth functions $u_I$ on $V$ such that, for any $\varphi\in \mathcal C^\infty_c(E)$
\[ ( U, \varphi)  = \int_V  \sum_Iu_I(v) D_I \varphi(v)\, dv\ .
\]
\begin{proof} By the previous result, $U = \sum_I D_I\widetilde u_I$, where $u_I$ is some distribution on $V$. The assumption on  the wavefront set of $U$, when tested against the functions of the form given by \eqref{tf} implies that, for each $d$-tuple $I$, $WF(\widetilde u_I)\subset V\times (V^\perp\setminus 0)$, which implies that $WF(u_I)=\emptyset$ by Proposition \ref{pwf}. As the projection onto the first coordinate of the wavefront set is precisely the singular support,
each $u_I$ coincides with a smooth function on $E$.
\end{proof}
\end{proposition}
A \emph{transverse differential operator} $D$ is a mapping form $\mathcal C_c^\infty(E)$ in $\mathcal C^\infty_c(V)$ which is given by
\[D\varphi(v) = \sum_I a_I(v) D_I\varphi(v)\ ,
\]
where $D_I$ are the normal derivatives introduced earlier, and the $a_I$'s are smooth functions on $V$. The sum is always assumed to be locally finite. Notice that the $a_I$ are well determined, again by testing the operator against functions of the form given by \eqref{tf}. The previous proposition can be reformulated as : any distribution $U$ supported on a linear subspace $V$, such that $WF(U) \subset V\times (V^\perp\setminus 0)$ can be realized as
\begin{equation}
( U, \varphi) = \int_V D\varphi(v) dv\ ,
\end{equation}
for some transverse differential operator $D$. Moreover, (once a splitting of $E$ as $V\oplus W$ has been chosen) $D$ is uniquely determined .

Invariance properties of a singular distribution are reflected in the associated transverse differential operator. Here is a special case, fitted for our needs.

\begin{proposition}\label{transinv} Let $U$ be in $\mathcal D'(E)$, supported on $V$. Assume that $U$ is invariant under translations by elements of $V$. Then \[\big( U, \varphi\big) = \int_V D\varphi(v) dv\ ,
\]
where $D$ is a transverse differential operator \emph{with constant coefficients}.
\end{proposition}
\begin{proof}
Let $v$ be any element of $V\setminus 0$, and let $X_v$ be the vector field on $E$ which is constant and equal to $v$ at each point of $E$. Then, the invariance property of $U$ amounts to the equalities $X_vU = 0$ for any $v\in V$. Hence, by \cite{h} Theorem 8.3.1

\[WF(U) \subset \{ (x,\xi)\in V\times (E'\setminus 0) , \xi(v) = 0\}\ .\]
As this is valid for any $v\in V$, 
\[WF(U) \subset \{ (x,\xi)\in V\times (V^\perp\setminus 0)\} .
\]
By Proposition \ref{smooth}, $U$ is given by a transverse differential operator $D$, i.e.
\[( U, \varphi) = \int_V D\varphi(v) dv\]
where $D\varphi(v) =\sum_I a_I(v) D_I\varphi(v)$. Now, for any $v_0\in V$,  $X_{v_0}$ commutes with any $D_I$, such that, by integration by parts, 
\[0 = ( U,X_{v_0} \varphi ) = -\int_V \sum_I X_{v_0}a_I(v) D_I\varphi(v) dv\ .
\]
Now fix a $d$-tuple $I$, and check this equality on functions of the form \eqref{tf}. It yields $X_{v_0} a_I = 0$ for any $v_0 \in V$ and hence $a_I $ is a constant.
\end{proof}

\medskip
\noindent
{\bf Remark}. All the results of this section could be formulated  for distributions  supported on a closed submanifold.

\section{An elementary approach to residues at poles of the second type}

In this section, we consider the standard Euclidean space $E=\mathbb R^d$ and denote the distance of two points $x,y\in E$ by $\vert x-y\vert$. Let $\beta$ be a complex number and let
\[l_\beta (x,y) = \vert x-y\vert^\beta\ .
\]
For $\Re \beta$ large enough, the kernel $l_\beta$ is locally integrable on $E \times E$, so that it defines a distribution on $E\times E$. It can be extended meromorphically (as a distribution), having simple poles at the points $\beta = -d-2k, k\in \mathbb N$ (see e.g. \cite{gs}).

Now let $\boldsymbol \beta = (\beta_1,\beta_2,\beta_3)$ be in $\mathbb C^3$. Set, for $x_1,x_2,x_3\in E$
\[l_{\boldsymbol \beta} (x_1,x_2,x_3)= l_{\beta_1}(x_2,x_3)\,l_{\beta_2}(x_3,x_1)\,l_{\beta_3}(x_1,x_2)\ .\]
For $f$ in $\mathcal C_c^\infty(E\times E\times E)$ the integral formula
\begin{equation*}
{\mathcal L}_{\boldsymbol \beta} (f) = \int_{E\times E\times E}\!\!\!\!\!\! \!\!\!\!\!\!f(x_1, x_2,x_3)\,l_{\boldsymbol \beta}(x_1,x_2,x_3)\, dx_1\,dx_2\,dx_3
\end{equation*}
is well defined  for $\Re (\beta_j) (j=1,2,3)$ large enough and can be meromorphically continued in $\mathbb C^3$ as a distribution on $E\times E\times E$. 

\begin{theorem}
The map $\boldsymbol \beta \longmapsto \mathcal L_{\boldsymbol \beta}$ can be meromorphically extended to $\mathbb C^3$, with simple poles along the four families of planes defined by one of the following equations
\begin{equation}\label{eqplane1}
\beta_1 = -d-2k_1,\quad \beta_2 = -d-2 k_2,\quad \beta_3 = -d-2k_3\end{equation}
\begin{equation}\label{eqplane2}
\beta_1+\beta_2+\beta_3 =-2d-2k
\end{equation}
where $k_1,k_2,k_3,k\in \mathbb N$.
\end{theorem}

The analoguous result for the sphere was obtained in a joint work of the second author with B. \O rsted \cite{co}, and the proof for the flat case requires only minor modifications. In fact the flat case and the spherical case are related by a stereographic projection already used in \cite{co}.

A pole is said to be \emph{of first type} if it satisfies one of the equations \eqref{eqplane1}, of \emph{second type} if it satisfies one of the equations \eqref{eqplane2}.

For poles of the first type,  the residues  were studied  in \cite{c}. We concentrate on poles of the second type. For each $k\in \mathbb N$, let $\mathcal H_k$ be the plane  defined by
\[ \mathcal H_k = \{ \boldsymbol \beta \in \mathbb C^3 \ ;\  \beta_1+\beta_2+\beta_3 = -2d -2k\}\ .\]

 In \cite{co} was already proven that, at a \emph{generic} point (see precise statement below) of such a plane,  the residue (viewed as a distribution on $E\times E\times E$) is supported on the diagonal 
\[\mathcal D= \{ (x,x,x)\vert x\in E\}\ .
\]
This will be reproved, in a more elementary and explicit way (see Theorem \ref{thres2}).

The distribution $\mathcal L_{\boldsymbol \beta}$ is invariant by diagonal translations, i.e. by mappings $t_v, v\in E$, where
\[t_v(x_1,x_2,x_3) = (x_1+v,x_2+v, x_3+v)\ .
\]
Clearly the residue at some pole $\boldsymbol \beta^0$ will have the same invariance. So, for $\boldsymbol \beta^0$ a pole of the second type, we are in the situation of Proposition \ref{transinv}. One has to choose a normal space to $\mathcal D$ in $E\times E\times E$, and our choice will be\begin{footnote}{As symmetry among the three variables is broken, from  now on, we use $(x,y,z)$ for notation of an element in $E\times E\times E$.}\end{footnote}
\[W= \{(0,y,z), y\in E, z\in E\}\ .
\]
The concept of transverse differential operator introduced in the previous section can be reinterpreted in this context. A \emph{bidifferential operator} $D$ is a map from smooth functions on $E\times E$ to smooth functions on $E$ of the form
\[D\varphi(v) = \sum_{I,J} a_{IJ}(v) \frac {\partial^{\vert I \vert + \vert J\vert} \varphi}{\partial y_I \partial z_J} (v,v)\ ,
\]
where $I$ and $J$ are $d$-tuples of integers, $y_j$ (resp $z_j$) are coordinates on the first factor (resp. second factor), associated to a choice of a basis of $E$ and the $a_{IJ}$'s are smooth functions on $E$. The sum is assumed to be locally finite. When the $a_{IJ}$ are constant functions (hence all $0$ except for a finite number), the operator is said to be with constant coefficients. 

\begin{theorem}\label{resbi}
 Let $\boldsymbol \beta^0\in \mathcal H_k$ for some $k\in \mathbb N$, but such that none of the equations \eqref{eqplane1} is satisfied. Then there exists a unique bidifferential operator with constant coefficients $D_{\boldsymbol \beta^0}$ such that, for $f$ in $\mathcal C^\infty_c(E)$ and $g\in \mathcal C^\infty_c(E\times E)$
\begin{equation}\label{bidiffres}
Res(\mathcal L_{\boldsymbol \beta}, \beta^0) (f\otimes g) = \int_E f(x) D_{\boldsymbol \beta^0} g(x) dx\ .
\end{equation}
\end{theorem} 

Here $f\otimes g$ stands for the function \[E\times E\times E\ni(x,y,z)\longmapsto f(x)g(x,y)\ .\]
Needless to say, as those functions are dense in $\mathcal C_c^\infty(E\times E\times E)$,  \eqref{bidiffres} determines completely the residue.

As already observed, the distribution $\mathcal L_{\boldsymbol \beta}$ is invariant by any diagonal translation. To take adavantage of this, define for $\varphi\in \mathcal C^\infty_c(E\times E\times E)$
\begin{equation}\label{Phi}
\Phi(y,z) = \int_E \varphi(v, y+v,z+v)\, dv \ .
\end{equation}
For $\varphi$ in $ \mathcal C^\infty_c(E\times E \times E)$, the integral converges and defines a function $\Phi$ which belongs to $ \mathcal C^\infty_c(E\times E)$. Moreover, the correspondance $\varphi \longmapsto \Phi$ is continuous. Notice for further reference that
\begin{equation}\label{Phi0}
\Phi(0,0) = \int_E \varphi(x,x,x) \,dx\ .
\end{equation}

\begin{lemma}
Assume that $\Re(\beta_j)>-d, j=1,2,3$ and $\Re(\beta_1+\beta_2+\beta_3) > -2d$. Then, for any $\varphi\in \mathcal C^\infty_c(E\times E\times E)$
\[\mathcal L_{\boldsymbol \beta} (\varphi) = \int_{E\times E} \vert y\vert ^{\beta_3}\vert z\vert^ {\beta_2}\vert y-z\vert^{\beta_1}\Phi(y,z)\,dy\,dz\ .
\]
\end{lemma}

\begin{proof} The conditions on the parameter $\boldsymbol \beta$ guarantee the convergence of the integrals. The equality of the integrals is obtained through the affine change of variables
\[v=x_1,\quad  y = x_2-x_1,\quad z = x_3-x_1\ .
\]
\end{proof}
Let 
\[\Sigma = \{ (\sigma, \tau) \in E \times E, \  \vert \sigma\vert^2+\vert \tau\vert^2 = 1\} \]
be the unit sphere in $E\times E$, and denote by $d\mu$ the Lebesgue measure on $\Sigma$. Recall the integration formula in polar coordinates
\begin{equation}
\int_{E\times E} \Phi(y,z) \,dy\, dz =\int_0^\infty \int_\Sigma \Phi(r\sigma, r\tau) d\mu(\sigma, \tau)\, r^{2d-1}\, dr
\end{equation}
\begin{lemma}\label{ibeta} 
For $\psi$ in $\mathcal C^\infty_c(\Sigma)$ let
\begin{equation}\label{psi}
\mathcal I_{\boldsymbol \beta}(\psi) =\int_\Sigma \vert \sigma\vert^{\beta_3} \vert \tau\vert^{\beta_2} \vert \sigma-\tau\vert^{\beta_1} \psi(\sigma, \tau)\, d\mu(\sigma ,\tau) \ .
\end{equation}

$i)$ Assume that  $\Re(\beta_j)>-d$ for $j=1,2,3$. Then the integral \eqref{psi} is convergent and defines a distribution $\mathcal I_{\boldsymbol \beta}$. 

$ii)$ The map $\boldsymbol \beta\longmapsto \mathcal I_{\boldsymbol \beta}$ can be extended meromorphically to $\mathbb C^3$, with simple poles along the family of planes given by the following equations :
\[  \beta_j = -d-2k_j, \quad j=1,2,3,\  k_j\in \mathbb N\ .
\]
\end{lemma}

\begin{proof}
The three subsets of $\Sigma$
\[\{ (\sigma, \tau)\in \Sigma, \sigma = 0\},\quad\{(\sigma, \tau)\in \Sigma, \tau = 0\}, \quad \{ (\sigma, \tau) \in \Sigma, \sigma = \tau\}
\]
are three \emph{disjoint} submanifolds of dimension $d-1$ (hence of codimension 
 $d$) in $\Sigma$. Recalling that we assumed  $\Re(\beta_j) > -d$, for $j=1,2,3$, the integrals
 \[\int_\Sigma \vert \sigma\vert^{\beta_3} d\mu(\sigma, \tau), \quad \int_\Sigma \vert \tau\vert^{\beta_2} d\mu(\sigma, \tau), \quad \int_\Sigma \vert \sigma-\tau\vert^{\beta_1} d\mu(\sigma, \tau)
  \]
 are convergent and hence the integral $\mathcal I_{\boldsymbol \beta}$ is convergent by applying a suitable argument involving a partition of unity. This shows $i)$. Similarly, the meromorphic extension and the location of poles (also the fact that the poles are simple) are classical and can be easily deduced from \cite{gs}.
\end{proof}
Let $\Phi$ be a function in $\mathcal C^\infty_c(E\times E)$. For $r$ in $\mathbb R$, let $\psi_r$ be the function on $\Sigma$ defined by
\begin{equation}\label{psir}
\psi_r(\sigma, \tau) = \Phi(r\sigma, r\tau), \quad (\sigma, \tau)\in \Sigma\ .
\end{equation}
Then $\psi_r$ belongs to $\mathcal C^\infty(\Sigma)$ and the map $(r,\Phi)\longmapsto \psi_r$ is continuous from $\mathbb R\times \mathcal C^\infty_c(E\times E)$ to $\mathcal C^\infty(\Sigma)$.

\begin{lemma} Assume that  $\Re(\beta_j)>-d, j=1,2,3$ and $\Re(\beta_1+\beta_2+\beta_3) > -2d$. Then, for any $\varphi\in \mathcal C^\infty_c(E\times E\times E)$
\begin{equation}\label{radint}
{ \mathcal L}_{\boldsymbol \beta}(\varphi) = \int_0^\infty r^{2d-1+\beta_1+\beta_2+\beta_3}\, \mathcal I_{\boldsymbol \beta} \psi_r \,dr\ .
\end{equation}
\end{lemma}
 This is just using the formula for integration in polar coordinates.
 
 \begin{lemma} Let $\gamma$ be in $\mathcal C^\infty_c(\mathbb R)$ and assume that $\gamma$ is an \emph{even} function. Then the integral
 \[I_s(\gamma)= \int_0^\infty r^{s} \gamma(r) dr = \frac{1}{2} \int_{-\infty}^{+\infty} \vert r\vert^s \gamma(r) dr
 \] is convergent for $\Re s >-1$.  The map $s\longmapsto I_s(\gamma)$ can be extended meromorphically to $\mathbb C$ with simple poles at $s=-1-2k, k\in \mathbb N$. Moreover, the residues at the poles are given by
 \begin{equation}\label{ress}
 Res(I_s(\gamma), -1-2k) = \frac{1}{\Gamma(2k+1)}{\big(\frac{d}{dr}\big)}^{2k}\!\!\!\gamma\ (0)\ .
 \end{equation}
\end{lemma}
 For a proof, see \cite{gs}.

\begin{theorem}\label{thres1}
 Let $k\in \mathbb N$ and let $\boldsymbol \beta^0 = (\beta^0_1,\beta^0_2, \beta^0_3)$ satisfy the following assumptions :
\medskip

$i)$ $\beta^0_1+\beta^0_2+\beta^0_3  = -2d -2k$
\medskip

$ii)$ $\beta^0_j\notin -d-2\mathbb N$.
\medskip

Let $\varphi$ be a function in $\mathcal C^\infty_c(E\times E\times E)$ and form successively the functions $\Phi$ defined by \eqref{Phi} and $\psi_r$ defined by \eqref{psir}. The 
function $\boldsymbol \beta \longmapsto \mathcal L_{\boldsymbol \beta}(\varphi)$ has a residue at $\boldsymbol \beta^0$ given by
\begin{equation}\label{resg}
Res(\mathcal L_{\boldsymbol \beta}(\varphi), \boldsymbol \beta^0)=
\frac{1}{\Gamma(2k+1)} {\big(\frac{d}{dr}\big)}^{2k}_{\vert r=0}\,{ \mathcal{I}}_{\boldsymbol \beta^0}(\psi_r)\ .
\end{equation}
\end{theorem}
\begin{proof}
Observe that $\mathcal I_{\boldsymbol \beta} \psi_r$ is well defined (lemma \ref{ibeta}) and, as a function of $r$ is easily seen to be in $\mathcal C^\infty_c(\mathbb R)$. Moreover, the distribution $\mathcal I_{\boldsymbol \beta}$ is even, whereas $\psi_{-r}(\sigma, \tau) = \psi_r(-\sigma,-\tau)$, hence $\mathcal I_{\boldsymbol \beta} \psi_r$ is an even function of $r$. Now let $\gamma = \mathcal I_{\boldsymbol \beta} \psi_r$. Then \eqref{radint} can be rewritten as
$\mathcal L_{\boldsymbol \beta} (\varphi) = I_s(\gamma)
$.  Observe that $2d-1+\beta^0_1+\beta^0_2+\beta^0_3 = -1-2k$ and eventually apply \eqref{ress} to conclude.
\end{proof}

The expression obtained for the residue (viewed as a distribution on $E\times E\times E$) shows that it is supported by the diagonal $\mathcal D$. In fact, if $\varphi\in \mathcal C^\infty(S)$ vanishes on a neighbourhood of $\mathcal D$, then $\Phi$ vanishes in a neigbourhood of $(0,0)$ in $E\times E$, and hence $\psi_r$ vanishes identically for $\vert r\vert$ small enough, so that $\mathcal I_{\boldsymbol \beta}(\psi_r)=0$ for $\vert r\vert$ small enough. Hence the residue (evaluated against $\varphi$) at $\boldsymbol \beta^0$ vanishes. Now, the structure of distributions supported by a submanifold is known from Schwartz's theorem. This requires choosing at each point $(x,x,x)$ of $\mathcal D$ a complementary subspace to the tangent space ("normal coordinates") at $\mathcal D$. The choice will be
\[ \mathcal N_{x,x,x} = \{(x,y,z), y\in E, z\in E\}\ .
\]
Moreover, as we are interested in the trilinear form rather than the distribution, the space of test functions will be restricted to functions of the form $(f\otimes g)\,(x,y,z) = f(x) g(y,z), x\in E, (y,z)\in E\times E$. 

With this change of point of view and notation, let us write more explicitly \eqref{resg}. We will use the following convention for coordinates on $E\times E$ : for $1\leq i\leq 2d$, let
\[v_i = y_i, {\rm \ if\ }1\leq i\leq d,\qquad v_i = z_{i-d}, {\rm \ if\ } d+1\leq i\leq 2d\ .
\]
and similarly on $\Sigma$
\[\rho_i = \sigma_i, {\rm \ if\ }1\leq i\leq d \qquad \rho_i = \tau_{i-d}, {\rm \ if\ } d+1\leq i\leq 2d\ .
\]

\begin{theorem}\label{thres2}
 Let $\boldsymbol \beta^0$ satisfy the same assumptions as in Theorem \ref{thres1}. Let $f\in \mathcal C^\infty_c(E)$ and $g\in \mathcal C^\infty_c(E\times E)$. Then
\begin{equation}
Res(\mathcal L_{\boldsymbol \beta}(f\otimes g), \boldsymbol \beta^0)=
\frac{1}{\Gamma(2k+1)}  \int_E f(x) (D_{\boldsymbol \beta^0} g)(x) dx\ ,
\end{equation}
where $D_{\boldsymbol \beta^0}$ is the bidifferential operator with constant coefficients given by
\[D_{\boldsymbol \beta^0} g (v) = \sum_{1\leq i_1, i_2,\dots, i_{2k}\leq 2d} a_{i_1,i_2,\dots, i_{2k}}(\boldsymbol {\beta^0})\frac{\partial^{2k} g} {\partial v_{i_1}\partial v_{i_2} \dots \partial v_{i_{2k}}}(v,v)\]
where
\begin{equation}\label{integrals}
a_{i_1,i_2,\dots, i_{2k}}(\boldsymbol {\beta^0}) = \int_\Sigma \rho_{i_1}\dots \rho_{i_{2k}} {\vert}\sigma\vert^{\beta^0_3}\vert \tau\vert^{\beta^0_2} \vert \sigma-\tau\vert ^{\beta^0_1}d\mu(\sigma, \tau)\ .
\end{equation}
\end{theorem}
\begin{proof}
First, for $(\sigma, \tau)$ in $\Sigma$,
\[ \psi_r(\sigma, \tau)= \int_E f(v) g(v+r\sigma,v+r\tau) dv
\]
so that
\[{\big(\frac{d}{dr}\big)}^{2k}_{\vert r=0}\psi_r(\sigma, \tau) = \int_E f(v) {R}^{(2k)}_{\sigma, \tau} g(v) dv\ ,
\]
where $R_{\sigma, \tau}^{(2k)}$ is the bidifferential operator given  by
\[R_{\sigma, \tau} ^{(2k)}\,g  (v)= \sum_{1\leq i_1, i_2,\dots, i_{2k}\leq 2d} \rho_{i_1}\rho_{i_2}\dots \rho_{i_{2k}}\frac{\partial^{2k}\ g\ }{\partial v_{i_1}\partial v_{i_2} \dots \partial v_{i_{2k}}}(v,v)\ .
\]
Hence
\[
Res(\mathcal L_{\boldsymbol \beta}(f\otimes g), \boldsymbol \beta^0) ={\big(\frac{d}{dr}\big)}^{2k}_{\vert r=0}\,{ \mathcal{I}}_{\boldsymbol \beta^0}(\psi_r)=
{ \mathcal{I}}_{\boldsymbol \beta^0} \Big({\big(\frac{d}{dr}\big)}^{2k}_{\vert r=0} \psi_r \Big)
\]
\[ = {\mathcal{I}}_{\boldsymbol \beta^0}\big(\int_E f(v) R_{\sigma,  \tau}^{(2k)}g (v) dv\big)=
\int_E f(v) D_{\boldsymbol \beta^0} g(v) dv
\]

where $D_{\boldsymbol \beta^0}$ is the bidifferential operator given by
\[D_{\boldsymbol \beta^0} g (v) = \sum_{1\leq i_1, i_2,\dots, i_{2k}\leq 2d} a_{i_1,i_2,\dots, i_{2k}}({\boldsymbol \beta^0})\frac{\partial^{2k} g} {\partial v_{i_1}\partial v_{i_2} \dots \partial v_{i_{2k}}}(v,v)\]
where
\begin{equation*}
a_{i_1,i_2,\dots, i_{2k}}(\boldsymbol {\boldsymbol \beta^0}) = \int_\Sigma \rho_{i_1}\dots \rho_{i_{2k}}{\vert} \sigma\vert^{\beta^0_3}\vert \tau\vert^{\beta^0_2} \vert \sigma-\tau\vert ^{\beta^0_1}d\mu(\sigma, \tau)\ ,
\end{equation*}
with the same convention as above.

Of course these integrals are to be understood in the sense of distributions, obtained by meromorphic continuation.
 \end{proof}
When $k=0$, there is only one term, so $D_{\boldsymbol \beta^0 }= c_{\beta^0} \Id$ and it is possible to evaluate the constant $c_{\beta^0}$.

\begin{proposition}\label{res0}
Assume that $\beta^0_1+\beta_2^0+\beta_3^0 = -2d$, and assume that $\beta^0_j\notin -d-2\mathbb N, j=1,2$ or $3$. Then
\begin{equation}
Res(\mathcal L_{\boldsymbol \beta}(f), \boldsymbol \beta^0)=c_0(\boldsymbol \beta^0)\int_E f(x)dx
\end{equation}
where
 \begin{equation}
c_0(\boldsymbol \beta^0) =\frac{\pi^d}{(2\sqrt 2)^{3d}}\,\frac{\Gamma(2d)}{\Gamma(\frac{3d}{2})}\,\frac{ \Gamma(\frac{\beta^0_1+d}{2})}{\Gamma(\frac{-\beta^0_1-d}{2})}\,\frac{ \Gamma(\frac{\beta^0_2+d}{2})}{\Gamma(\frac{-\beta^0_2-d}{2})}\,\frac{ \Gamma(\frac{\beta^0_3+d}{2})}{\Gamma(\frac{-\beta^0_3-d}{2})}\ .
\end{equation}
 \end{proposition}
 \begin{proof}
 Let $c$ be the stereographic projection, defined from $E$ into the sphere $S$ of radius $1$ in $\mathbb R^{d+1}$ by
 \[c(x) = \begin{pmatrix} \frac{1-\vert x\vert^2}{1+\vert x\vert ^2}\\ \frac{2x_1}{1+\vert x\vert^2}\\ \vdots\\ \frac{2x_{d}}{1+\vert x\vert^2}\end{pmatrix}\ .
\]
 The map $c$ is conformal, and more precisely, for any tangent vector $\xi\in \mathbb R^d$,
 \[\vert Dc(x)\xi \vert= \frac{2}{1+\vert x\vert^2}\vert \xi\vert
 \]
 which implies
 \begin{equation*}
 \int_S f(\sigma) d\sigma =  2^d\int_E f(c(x)) (1+\vert x\vert^2)^{-d}dx 
 \end{equation*}
 for $f$ any integrable function  on $S$.
 Moreover, for any $x, y$ in $E$,
 \[\vert c(x)-c(y)\vert = \frac{2\vert x-y\vert}{(1+\vert x\vert ^2)^{1/2}(1+\vert y\vert ^2)^{1/2}}\ .
 \]
By the same change of variables, the trilinear form $\mathcal  L_{\boldsymbol \beta}$ is related to the trilinear form $\mathcal K_{\boldsymbol \alpha}$ on $S$  studied in \cite{co} through the relation
\begin{equation}\label{stereo}
\mathcal K_{\boldsymbol \alpha}(1\otimes 1\otimes 1) = 2^{\alpha_1+\alpha_2+\alpha_3+3\rho} \mathcal L_{\boldsymbol \beta}(f_0)
\end{equation}
where 
$\alpha_j = \beta_j+\rho$ for $j=1,2,3$, and $f_0$ is the function on $E\times E\times E$ given by
\[f_0(x_1,x_2,x_3) =(1+\vert x_1\vert ^2)^{-\frac{\beta_2+\beta_3}{2}}(1+\vert x_2\vert ^2)^{-\frac{\beta_3+\beta_1}{2}}(1+\vert x_3\vert ^2)^{-\frac{\beta_1+\beta_2}{2}}\ .
\]
The left handside of \eqref{stereo} has been computed in \cite{de}, and by a different method in \cite{ckop} (see also \cite{co}), so that 
\begin{equation}
\mathcal L_{\boldsymbol \beta}(f_0) =\big( \frac{\sqrt \pi}{2\sqrt 2}\big)^{3d}\,\frac{ \Gamma(\beta_1+\beta_2+\beta_3+2d)\Gamma(\frac{\beta_1+d}{2})\Gamma(\frac{\beta_2+d}{2})\Gamma(\frac{\beta_3+d}{2})}
{\Gamma(\frac{\beta_2+\beta_3+d}{2})\Gamma(\frac{\beta_3+\beta_1+d}{2})\Gamma(\frac{\beta_1+\beta_2+d}{2})}
\end{equation}
Hence,
\[Res(\mathcal L_{\boldsymbol \beta}({f_0}), \boldsymbol \beta^0) =\big( \frac{\sqrt \pi}{2\sqrt 2}\big)^{3d}\frac{ \Gamma(\frac{\beta^0_1+d}{2})\Gamma(\frac{\beta^0_2+d}{2})\Gamma(\frac{\beta^0_3+d}{2})}
{\Gamma(\frac{\beta^0_2+\beta^0_3+d}{2})\Gamma(\frac{\beta^0_3+\beta^0_1+d}{2})\Gamma(\frac{\beta^0_1+\beta^0_2+d}{2})}\ .
\]
Now
\[\int_E f_0(x,x,x) dx = \int_E (1+\vert x\vert^2)^{-(\beta_1+\beta_2+\beta_3)}\,dx= \int_E (1+\vert x\vert^2)^{-2d}\, dx
\]
\[ = vol(S^{d-1})\int_0^\infty (1+r^2)^{-2d}\, r^{d-1} dr = \pi^{\frac{d}{2}}\,\frac{\Gamma(\frac{3d}{2})}{\Gamma(2d)}
\]
from which the lemma follows.
 \end{proof}
 
In the general case (when $k\geq 1$), the integrals seem difficult to compute. Instead,  in order to evaluate more explicitly the residue,  it is possible to use the stronger invariance properties of the trilinear forms, which translate  into \emph{covariance properties} of the bidifferential operator $D_{\boldsymbol \beta^0}$.
 
 \section{Conformally covariant bidifferential operators}

Introduce the group of conformal transformations of $E$. A local transformation $\Phi$ of $E$ is said to be conformal if, at any point $x$ where $\Phi$ is defined, and for any tangent vector $\xi$,

\[ \vert D\Phi(x)\xi \vert= \kappa(x) \vert \xi\vert\ .
\]
where $\kappa$ is a smooth strictly positive function, called the conformal factor of $\Phi$. Classically, to any element of the group $G=SO_0(1,d+1)$, one can attach  a rational conformal action on $E$. If $d\geq 3$, then this group exhausts the group of positive local conformal diffeomorphisms (Liouville's theorem). The group $G$ operates globally on the sphere $S=S^d$ of dimension $d$ (see \cite{tak}) and this action can be transferred to a (not everywhere defined) action on $E$ by using the stereographic projection. It can also be realized as the group generated by the translations, the rotations, the dilations and the symmetry-inversion $\iota$
\[ \iota : x\longmapsto -\frac{x}{\vert x\vert^2}\ .\]
To any element $g\in G$, let $\kappa(g,x)$ be its conformal factor. Then $\kappa$ satisfies a cocycle relation, namely
\[
\kappa(g_1g_2,x) = \kappa(g_1,g_2(x))\kappa(g_2,x)\ .
\]
A family of representations is associated to that cocycle. For $\lambda$ in $\mathbb C$, define
\[\pi_\lambda(g) f(x) = \kappa(g^{-1},x)^{\rho+\lambda} f(g^{-1}(x))\ ,
\]
where we set $\rho=\frac{d}{2}$. These representations are the noncompact realization of the principal spherical series of $SO_0(1,d+1)$ (cf \cite{tak}).

Let $\boldsymbol \lambda = (\lambda_1,\lambda_2,\lambda_3)\in \mathbb C^3$. A continuous trilinear form $\mathcal L $ on $ \mathcal C^\infty_c(E)\times  \mathcal C^\infty_c(E)\times \mathcal C^\infty_c(E)$ is said to be \emph{conformally invariant} with respect to $(\pi_{\lambda_1}, \pi_{\lambda_2}, \pi_{\lambda_3})$ if, for three functions $f_1,f_2,f_3$ in $\mathcal C^\infty_c(E)$, for any $g$ in a (sufficently small) neighbourhood of the neutral element in $G$,
\begin{equation}
\mathcal L\big(\pi_{\lambda_1}(g)f_1\otimes \pi_{\lambda_2}(g)f_2\otimes \pi_{\lambda_3}(g)f_3\big)
=\mathcal L\big(f_1\otimes f_2\otimes f_3\big)\ .
\end{equation}

Recall the main result of \cite{co}, which was stated for the action of the group $G$ on the sphere, but the situations are essentially equivalent through a stereographic projection.

\begin{proposition} Let $\boldsymbol \beta\in \mathbb C^3$ and assume that none of the conditions \eqref{eqplane1}, \eqref{eqplane2} is satisfied. Let $\boldsymbol \lambda = (\lambda_1,\lambda_2,\lambda_3)$ be the unique element of $\mathbb C^3$ defined by the equations
\begin{equation}\label{betalambda}
\begin{split}
\beta_1 =& -\lambda_1+\lambda_2+\lambda_3-\rho\\
\beta_2 =&\ \lambda_1-\lambda_2+\lambda_3-\rho\\
\beta_3 =&\ \lambda_1+\lambda_2-\lambda_3-\rho\ .
\end{split}
\end{equation}
$i)$ The trilinear form $\mathcal L_{\boldsymbol \beta}$ is conformally invariant with respect to $\pi_{\lambda_1}\otimes \pi_{\lambda_2} \otimes \pi_{\lambda_3}$.

\noindent
$ii)$ Any continuous trilinear form which is conformally invariant with respect to $\pi_{\lambda_1}\otimes \pi_{\lambda_2} \otimes \pi_{\lambda_3}$ is proportional to $\mathcal L_{\boldsymbol \beta}$.
\end{proposition}

By analytic continuation, the invariance is also valid for the residue of $\mathcal L_{\boldsymbol \beta}$ at some pole $\boldsymbol \beta^0$. At poles of second type, this invariance property can in turn be translated in a covariance property for the bidifferential operator $D_{\boldsymbol \beta^0}$.

\begin{definition} Let $D$ be a bidifferential operator from $\mathcal C^\infty_c(E\times E)$ into $\mathcal C^\infty_c(E)$. Let $\lambda, \mu, \nu$ be three complex numbers. Then $D$ is said to be \emph{covariant with respect to $(\pi_\lambda\otimes \pi_\mu, \pi_\nu)$} if for any functions $f \in \mathcal C^\infty_c(E\times E)$ and $g$ in a (small enough) neighbourhood of the neutral element in $G$,
\begin{equation*}
D(\pi_\lambda(g)\otimes \pi_\mu(g) f) = \pi_\nu(g)(Df)
\end{equation*}
\end{definition}

Recall the following duality result for the representations $\pi_\lambda$.

\begin{proposition} \label{dual}
Let $\lambda\in \mathbb C$. Then, for any functions $\varphi$ and $\psi$ in $\mathcal C^\infty_c(E)$ 
\[\int_E \pi_\lambda(g)\varphi (x) \psi(x) dx = \int_E\varphi(x) \pi_{-\lambda}(g^{-1}) \psi(x) dx\ .
\]
\end{proposition}

This duality result links together (singular) conformally invariant trilinear forms and covariant bidifferential operators.

\begin{proposition} Let $D$ be a bidifferential operator from $\mathcal C^\infty_c(E\times E)$ into $\mathcal C^\infty_c(E)$. Let $\mathcal L$ be the continuous trilinear form  defined for $f\in \mathcal C^\infty_c(E)$ and $g\in \mathcal C^\infty_c(E\times E)$ by
\[\mathcal L(f\otimes g) = \int_E f(x) Dg(x) dx\ .
\]
Let $\lambda, \mu, \nu$ be three complex numbers. Then the form $\mathcal L$ is invariant with respect to $(\pi_\lambda, \pi_\mu, \pi_\nu)$ if and only if $D$ is covariant with respect to $(\pi_\mu\otimes \pi_\nu, \pi_{-\lambda})$.
\end{proposition}

\begin{corollary} Let $\boldsymbol \beta= (\beta_1, \beta_2, \beta_3)\in \mathcal H_k$ for some $k\in \mathbb N$, and such that none of the conditions \eqref{eqplane1} is satisfied.  Let $D_{\boldsymbol \beta}$ be the associated bidifferential operator given by  Theorem \ref{thres2}. Let $\boldsymbol \lambda = (\lambda_1,\lambda_2,\lambda_3)$ be given by equations \eqref{betalambda}.
Then the bidifferential operator $D_{\boldsymbol \beta}$ is covariant with respect to $(\pi_{\lambda_2}\otimes \pi_{\lambda_3}, \pi_{\lambda_2+\lambda_3+\rho+2k})$
\end{corollary}

This is a consequence of the expression of the residue Theorem \ref{resbi}, together with the duality result (Proposition \ref{dual}).
The fact that none of the conditions \eqref{eqplane1} is satisfied amounts to the conditions 
\[\lambda_2, \lambda_3\notin -k+\mathbb N,\qquad \lambda_2+\lambda_3 \notin -\rho-k-\mathbb N\ .
\]

Covariant differential operators for the conformal group have been studied intensively, and the following result was obtained sometimes ago by V. Ovsienko and P. Redou (see \cite{or}). Recall the \emph{Pochhammer's symbol}, for $a$ a complex number and $m\in \mathbb N$
\[(a)_m = a(a+1)(a+2)\dots (a+m-1)\ .
\]

\begin{proposition} Let $k$ be a nonnegative integer, and let $\lambda, \mu$ be complex numbers.

$i)$ Assume that 
$ \lambda,\mu\notin
\{0, -1,-2,\dots, -(k-1)\}$.
Then there exists a  bidifferential operator $D_{\lambda, \mu}^{(k)}$ which is covariant with respect to $(\pi_\lambda\otimes \pi_\mu, \pi_{\lambda+\mu+\rho+2k})$.

$ii)$ Assume moreover that $\lambda, \mu \notin \{-\rho, -\rho-1,\dots, -\rho-(k-1)\}$. Then the operator is unique up to a constant.
\end{proposition}

The operator $D_{\lambda ,\mu}^{(k)}$ is explicitly described.  The three fundamental bidifferential operators are $\Delta_y, \Delta_z$ and the operator $R$ defined  for $f$ in $\mathcal C^\infty(E\times E)$ by
\[ R(f)(x) = \sum_{j=1}^d \frac{\partial ^2 f}{\partial y_j \partial z_j}(x,x)\ .
\]
Then 
\[D_{\lambda ,\mu}^{(k)}= \sum_{r,s,t, r+s+t =k} c_{rst} \Delta_y^rR^s \Delta_z^t
\]
where the $c_{rst}$ are explicitly determined coefficients, depending on $\lambda, \mu$ and $k$, namely
\begin{equation}
\begin{split}
c_{rst} =& \frac{(-1)^{t-r}}{2^r\,r!}\begin{pmatrix}r+s+t\\t\end{pmatrix}\frac{(s+1)_r}{(\lambda+1)_r}
\\
&\sum_{p=0}^r\frac{r!t!}{p!}\,\frac{(\lambda+\rho+r-s+p)_{t-p}\,(\mu+\rho+s+2t)_{r-p}}{({\mu+1})_{t-p}}
\end{split}
\end{equation}
when $r\leq t$, and for $r\geq t$, $c_{rst}(\lambda,\mu) = c_{tsr}(\mu,\lambda)$.

\begin{theorem}\label{resor}
 Let $\boldsymbol \beta\in \mathcal H_k$ for some $k$ in $\mathbb N$. Assume moreover that none of the conditions \eqref{eqplane1} are satisfied. Let $\boldsymbol \lambda = (\lambda_1,\lambda_2,\lambda_3)$ the unique solution of the system \eqref{betalambda}. Then $(\lambda_2,\lambda_3)\notin \{0,-1,-2,\dots -(k-1)\}$ and \begin{equation}
D_{\boldsymbol \beta} = c_{\boldsymbol \beta} D_{\lambda_2,\lambda_3}^{(k)}
\end{equation}
for some constant $c_{\boldsymbol \beta}$.
\end{theorem}

\begin{proof}
The conditions on $\boldsymbol \beta$ imply that 
 $\lambda_2,\lambda_3\notin -k+\mathbb N$, so that the conditions for the defintion of $D_{\lambda,\mu}^{(k)}$ are satisfied. Assume for a while that $\boldsymbol \beta$ is such that
  $\lambda_2,\lambda_3\notin \{0, -\rho,-\rho-1,\dots, -\rho-(k-1)\}$. By the uniqueness statement, there exists a constant $c_{\beta}$ such that $D_{\boldsymbol \beta} = c_{\boldsymbol \beta} D_{\lambda_2,\lambda_3}^{(k)}$. By analytic continuation in the plane $\mathcal H_k$, this equality remains valid on the domains were both sides are defined.
\end{proof}
For instance, if $k=1$,
\begin{equation}\label{resh1}
D_{\lambda,\mu}^{(1)} = -\frac{\mu+\rho}{\lambda+1}\,\Delta_y +2 R -\frac{\lambda+\rho}{\mu+1}\, \Delta_z\ .
\end{equation}

Theorem \ref{resor} however does not determinate the constant $c_{\boldsymbol \beta}$, and it seems quite difficult to test the operator $D_{\lambda_2,\lambda_3}^{(k)}$ against the function $f_0$ as we did for the determination of the residue at a pole in $\mathcal H_0$. 

\section{Bernstein-Sato identity}

A \emph{Bernstein-Sato identity} (on the first parameter) is an identity of the form
\[B \,\ell_{\boldsymbol \beta+2_1}= b(\boldsymbol \beta) \ell_{\boldsymbol \beta}\ ,
\]
where $\boldsymbol \beta +2_1 = (\beta_1+2,\beta_2,\beta_3)$, $B = B((x,y,z),\partial_x, \partial_y, \partial_z, {\boldsymbol \beta})$ is a differential operator with polynomial coefficients on $E\times E\times E$ and depending polynomially in $\boldsymbol \beta$, and $b$ is a polynomial in three complex variables. Such identities are known to exist (see \cite{sab1}, \cite {sab2}), but are in general very difficult to find. It turns out that, in the case at hand, it is possible to find such identities. The proof uses in a crucial way the covariance property of the kernel $l_{\boldsymbol \beta}$ with respect to the conformal action of $G$ on $E$. 

During the proof of some results, we will need to use the Euclidean Fourier transform. So it requires to extend the trilinear form to the Schwartz space $\mathcal S(E\times E\times E)$. This is merely routine. For the definition of invariance, one should formulate the condition in terms of the infinitesimal action of the conformal group. It is a classical computation (see e.g. \cite{or}) and the action of the Lie algebra $\mathfrak g = so(1,d+1)$ or of the universal enveloping algebra is by differential operators with polynomial coefficients. Hence they operate on the Schwartz space $\mathcal S(E)$. Moreover the meromorphic continuation of $\ell_{\boldsymbol \beta}$ yields tempered distributions (cf \cite{gs} for similar examples). Details are left to the reader.
\begin{lemma}\label{matrcoeff}
 Let $M$ be the operator on $\mathcal S(E\times E)$ given by
\[M\varphi(y,z) = \vert y-z\vert^2 \varphi(y,z)\ .
\]
Let $\lambda, \mu$ be two complex parameters.Then $M$ is an intertwining operator for $(\pi_\lambda\otimes \pi_\mu, \pi_{\lambda -1}\otimes \pi_{\mu-1})$.
\end{lemma}

\begin{proof} Let $g\in G$, and $\varphi \in \mathcal S(E\times E)$, and assume that $g$ is defined on a neighbourhood of $Supp(\varphi)$. Then
\[[M\circ \big(\pi_\lambda(g)\otimes \pi_\mu(g)\big)] \varphi(y,z) =\vert y-z\vert^2 \kappa(g^{-1}, y)^{\rho+\lambda} \kappa(g^{-1} ,z)^{\rho+\mu} \varphi(g^{-1}(y), g^{-1}(z))
\]
\[=\vert g^{-1}(y)-g^{-1}(z)\vert^2 \kappa(g^{-1}, y)^{\rho+\lambda-1} \kappa(g^{-1} ,z)^{\rho+\mu-1} \varphi(g^{-1}(y), g^{-1}(z))
\]
\[ =[ \pi_{\lambda-1}(g)\otimes \pi_{\mu-1}(g) ] (M\varphi)\ (y,z)\ .
\]
\end{proof}
Introduce now the Knapp-Stein intertwining operator. For $\nu$ a complex parameter, let $I_\nu$ be the operator on $\mathcal S(E)$ given by
\[I_\nu(f)(x) = \int_E \vert x-y\vert^{-d+\nu} f(y) dy\ .
\]
For $\Re \nu>0$, the integral is convergent and defines a continuous operator on $S(E)$. It can be meromorphically continued to $\mathbb C$, with simple poles at $\nu = -2k, k\in \mathbb N$. It satisfies the following intertwining property
\begin{equation}\label{intw}
I_{2\nu} \circ \pi_\nu(g) = \pi_{-\nu}(g) \circ I_{2\nu}\ .
\end{equation}
Now, for $\lambda, \mu$ two complex parameters, form the operator
 \[N_{ \lambda, \mu} = I_{-2\lambda-2}\otimes I_{-2\mu-2}\circ M\circ I_{2\mu}\otimes I_{2\mu}\]
\[\mathcal S(E\times E)\xrightarrow{I_{2\nu}\otimes I_{2\mu}} \mathcal S(E\times E)\xrightarrow{M} \mathcal S(E\times E)\xrightarrow {I_{-2\lambda-2}\otimes I_{-2\mu-2}}\mathcal S(E\times E)\ .
\]

For generic values of the parameters $(\lambda, \mu)$, $N_{\lambda, \mu}$ is a well defined operator on $\mathcal S(E\times E)$, which, by construction intertwines the representation $\pi_\lambda\otimes \pi_\mu$ and $\pi_{\lambda+1}\otimes \pi_{\mu+1}$.

Let $\boldsymbol \lambda = (\lambda_1,\lambda_2,\lambda_3)$ be a generic triple in $\mathbb C^3$, let $\boldsymbol \beta=(\beta_1,\beta_2,\beta_3)$ be the triplet associated to $\boldsymbol \lambda$ through \eqref{betalambda}. Observe that $\boldsymbol \beta +2_1$ is associated to the triple $(\lambda_1, \lambda_2+1,\lambda_3+1)$.

Consider the continuous trilinear form $\mathcal L$ on $\mathcal C^\infty_c(E)\times \mathcal C^\infty_c(E)\times \mathcal C^\infty_c(E)$ given by
\[ \mathcal L(f_1,f_2,f_3)  = \mathcal L_{\boldsymbol \beta+2_1}(f_1\otimes N_{\lambda_2,\lambda_3}(f_2\otimes f_3))\ .\]
From the intertwining property of $N_{\lambda_2,\lambda_3}$
\[\mathcal L(\pi_{\lambda_1}(g)f_1, \pi_{\lambda_2}(g)f_2,\pi_{\lambda_3} (g) f_3) = \mathcal L_{\boldsymbol \beta+2_1}(\pi_{\lambda_1}(g)f_1\otimes N_{\lambda_2, \lambda_3}[\pi_{\lambda_2}(g)f_2\otimes \pi_{\lambda_3}(g) f_3])
\]
\[ = \mathcal L_{\boldsymbol \beta+2_1}(\pi_{\lambda_1}(g)f_1\otimes [\pi_{\lambda_2+1}(g)\otimes \pi_{\lambda_3+1}(g)]\circ N_{\lambda_2,\lambda_3} [f_2\otimes f_3])
\]
\[ = \mathcal L_{\boldsymbol \beta+2_1}( f_1\otimes N_{\lambda_2,\lambda_3}( f_2\otimes f_3))
\]
\[ = \mathcal L(f_1,f_2,f_3)\ ,
\]
so that the form $\mathcal L$ is invariant w.r.t. $(\pi_{\lambda_1}, \pi_{\lambda_2}, \pi_{\lambda_3})$.
By the generic uniqueness result  on the invariant trilinear form (see \cite{co}), the form $\mathcal L$ has to be proportional to $\mathcal L_{\boldsymbol \beta}$. Hence there exists a constant $e=e(\boldsymbol \beta)$  such that
\begin{equation}\label{bsheur}
(N_{\lambda_2,\lambda_3})^t \big(l_{\boldsymbol \beta+2_1}\big) = e(\boldsymbol \beta) l_{\boldsymbol \beta}\ .
\end{equation}
As we will see now, the operator $N_{\lambda, \mu}$ (hence also its transpose) is a differential operator on $E\times E$, so that \eqref{bsheur} {\it is} indeed a Bernstein-Sato identity.

\begin{proposition}
For $\lambda,\mu \in \mathbb C^2$, let $E_{\lambda,\mu}$ be the differential operator on $E\times E$ defined by
\begin{equation}
\begin{split}
E_{\lambda, \mu} =&\ \vert y-z\vert^2 \Delta_{y}\Delta_{z}\\
&-4\mu\sum_{j=1}^d (z_j-y_j)\frac{\partial}{\partial z_j}\Delta_y
-4\lambda\sum_{j=1}^d (y_j-z_j)\frac{\partial}{\partial y_j}\Delta_z\\
&+2\mu(2\mu+2-d) \Delta_y
+2\lambda(2\lambda+2-d)) \Delta_z\\
&-8 \lambda\mu\sum_{j=1}^d\frac{ \partial}{\partial y_j}\frac{ \partial}{\partial z_j}\ .\\
\end{split}
\end{equation}
Its transpose $F_{\lambda,\mu}=E_{\lambda,\mu}^t$ is given by
\begin{equation}\label{flambdamu}
\begin{split}
F_{\lambda,\mu} =& \ \vert y-z\vert^2 \Delta_{y}\Delta_{z}\\
&+4(\mu+1) \sum_{j=1}^d (z_j-y_j)\frac{\partial}{\partial z_j}\Delta_y+4(\lambda+1) \sum_{j=1}^d (y_j-z_j) \frac{\partial}{\partial y_j}\Delta_z\\
&+4(\mu+1)(\mu+\rho) \Delta_y+4(\lambda+1)(\lambda+\rho)\Delta_z\\
&-8(\lambda+1)(\mu+1)\sum_{j=1}^d \frac{ \partial}{\partial y_j}\frac{ \partial}{\partial z_j}\\
\end{split}
\end{equation}
 The operator $N_{\lambda, \mu}$ (for generic $(\lambda, \mu)$) is a differential operator on $E\times E$, and is given by
 \[N_{\lambda, \mu} = c(\lambda, \mu) F_{\lambda, \mu}\ ,
 \]
 where
 \[c(\lambda, \mu) = \frac{\pi^{2d}}{16} \,\frac{\Gamma(\lambda)\Gamma(-\lambda-1) \Gamma(\mu)\Gamma(-\mu-1)}{\Gamma(\rho-\lambda)\Gamma(\rho+\lambda+1)\Gamma(\rho-\mu)\Gamma(\rho+\mu+1)}\ .
 \]
 
 \end{proposition}

\begin{proof} Introduce the Fourier transform on  $E$, defined by
\[\mathcal F f (\xi) = \hat f(\xi)= \int_{E\times E} e^{-i(\xi,x)}f(x) dx \ ,\]
and extend it by duality to $\mathcal S'(E)$. The Fourier transform on $E\times E$ is defined accordingly.
Observe that $I_\nu$ is a convolution operator with a tempered distribution, so that the Fourier transform of $I_\nu f$ is given by the product of the Fourier transform, i.e. 
\[\mathcal F(I_\nu f) ( \xi) = c(\nu) \vert \xi\vert^{-\nu} \hat f(\xi)
\]
where 
\[ c(\nu) = 2^\nu \pi^{\frac{d}{2}}\frac{\Gamma(\frac{\nu}{2})}{\Gamma(\frac{d-\nu}{2})}\]
(see e.g. \cite{gs}).

Next, as $M$ acts by multiplication by a polynomial, the Fourier transform of $M\varphi$ is given by
\[\mathcal F(M\varphi) (\xi, \eta) = (-\Delta_\xi +2 R - \Delta_\eta) \hat \varphi (\xi, \eta),
\]
where $\Delta$ is the Laplacian on $E$ and $R$ is the differential operator on $\mathcal S(E\times E)$ defined by
\[R \varphi (\xi, \eta) =  \sum_{j=1}^d \frac{\partial^2 \varphi}{\partial \xi_j\partial \eta_j}\ .
\]
To prove the formula, it is enough to prove it for functions $\varphi = f\otimes g$, where $f,g\in \mathcal S(E)$.
Now
\[\frac{\partial}{\partial \xi_j} (\vert \xi\vert^{-2\lambda} \hat f(\xi)) = \vert\xi\vert^{-2\lambda} \frac{\partial \hat f}{\partial \xi_j}-2\lambda \vert \xi\vert^{-2\lambda-2} \xi_j \hat f(\xi),
\]
so that
\begin{equation*}
	\begin{split}
	\Delta_\xi (\vert \xi\vert^{-2\lambda} \hat f(\xi)) &= \vert\xi\vert^{-2\lambda} \Delta \hat f (\xi) \\
	&\quad  -4\lambda\vert \xi\vert^{-2\lambda-2}\sum_{j=1}^d \xi_j \frac{\partial \hat f}{\partial \xi_j}\\&\quad + 2\lambda(2\lambda 
	 +2-d)\vert \xi \vert^{-2\lambda-2}\hat f(\xi)
	\end{split}
\end{equation*}
and
\begin{equation*}\begin{split}
R(\vert\xi\vert^{-2\lambda}  \hat f(\xi) \vert\eta\vert^{-2\mu}\hat g(\eta)) &=\ 
 \sum_{j=1}^d \vert \xi\vert^{-2\lambda} \frac{\partial \hat f}{\partial \xi_j}\,\vert\eta\vert^{-2\mu} \frac{\partial \hat g}{\partial \eta_j}\\
& \quad -2\lambda\sum_{j=1}^d  \xi_j \vert \xi\vert^{-2\lambda-2} \hat f(\xi) \vert\eta\vert^{-2\mu} \frac{\partial \hat g}{\partial \eta_j}\\
&\quad -2\mu \sum_{j=1}^d \vert \xi\vert^{-2\lambda} \frac{\partial \hat f}{\partial \xi_j} \eta_j  \vert \eta\vert^{-2\mu-2} \hat g (\eta)\\
& \quad +4\lambda\mu \vert \xi\vert^ {-2\lambda-2} \vert \eta\vert^{-2\mu-2}(\sum_{j=1}^d  \xi_j\eta_j)\hat f(\xi)\hat g(\eta).\\
\end{split}
\end{equation*}
Let apart the factor $ c(2\lambda)c(2\mu)c(-2\lambda-2)c(-2\mu-2)$, the Fourier transform of $N_{\lambda, \mu}(f\otimes g)$ is given by
\begin{multline*}
- \vert\xi\vert^2 \Delta_\xi \hat f (\xi) \vert\eta\vert^2\hat g(\eta)
+4\lambda\sum_{j=1}^d \xi_j \frac{\partial \hat f}{\partial \xi_j}(\xi)\vert \eta\vert^2\hat g(\eta)\\
- 2\lambda(2\lambda+2-d)\hat f(\xi)\vert \eta \vert^2\hat g(\eta)
+2 \sum_j\vert \xi\vert^2 \frac{\partial \hat f}{\partial \xi_j}(\xi)\vert \eta\vert^2\frac{\partial \hat g}{\partial \eta_j}(\eta)\\
-4\lambda  \sum_{j=1}^d \xi_j  \hat f(\xi)\vert \eta \vert^2\frac{\partial \hat g}{\partial \eta_j}(\eta)
-4\mu  \sum_{j=1} \vert \xi\vert^2\frac{\partial \hat f}{\partial \xi_j}(\xi) \eta_j  \hat g(\eta)\\
+8\lambda\mu \sum_{j=1}^d \xi_j\hat f(\xi) \eta_j \hat g(\eta)
-\vert   \xi\vert^2\hat f(\xi)\vert \eta\vert^2\Delta_\eta \hat g(\eta)\\
+4\mu\sum_{j=1}^d \vert\xi\vert^2 \hat f(\xi) \eta_j\frac{\partial \hat g}{\partial \eta_j}(\eta)
-2\mu(2\mu+2-d) \vert \xi\vert^2 \hat f(\xi)\hat g(\eta)\ .\\
\end{multline*}

Now use the classical formul\ae
\begin{align*}
\widehat{ \frac{\partial f}{\partial y_j}} (\xi) &= i\xi_j \hat f(\xi) &\widehat {( y_jf)} (\xi) &= i \frac{\partial \hat f}{\partial \xi_j} (\xi) \\
\widehat {\Delta f}(\xi) &= -\vert \xi\vert^2{\hat{f}(\xi)}  &\widehat{\vert y\vert^2f(y)} (\xi) &= -\Delta \hat f (\xi) 
\end{align*}
to obtain the following expression for $N_{\lambda,\mu}(f\otimes g)$ (up to the factor $ c(2\lambda)c(2\mu)c(-2\lambda-2)c(-2\mu-2)$) :

\begin{multline*}
\Delta_y (\vert y\vert^2 f) \Delta_zg
+4\lambda \sum_{j=1}^d \frac{\partial}{\partial y_j}(y_jf) \Delta_zg\\
+2\lambda(2\lambda+2-d) f \Delta_zg
-2\sum_{j=1}^d \Delta_y (y_jf) \Delta_z(z_jg)\\
-4\lambda \sum_{j=1}^d \frac{\partial}{\partial y_j} f \Delta_z (z_jg)
-4\mu \sum_{j=1}^d \Delta_y(y_jf) \frac{\partial}{\partial z_j} g\\
-8\lambda\mu \sum_{j=1}^d \frac{\partial}{\partial y_j} f \frac{\partial}{\partial z_j} g
+\Delta_yf \Delta_z(\vert z\vert^2 g)\\
+4\mu \sum_{j=1}^d \Delta_y f \frac{\partial}{\partial z_j} (z_j g)
+2\mu(2\mu+2-d) \Delta_y f \, g\\
\end{multline*}
The final expression for $N_{\lambda ,\mu} $ and $(N_{\lambda ,\mu})^t$ follows easily.
\end{proof}
As announced,  $E_{\lambda_2, \lambda_3}$ (being proportional to $N_{\lambda_2,\lambda_3}^t$)is a candidate for a Bernstein-Sato identity for the kernel $\ell_{\boldsymbol \beta}$ (the $\lambda$'s being related to $\boldsymbol \beta$ by \eqref{betalambda}). By brute force computation, the following identity is obtained.

\begin{theorem} [Bernstein-Sato identity]\label{bs} For $\boldsymbol \beta=(\beta_1,\beta_2,\beta_3)$ in $\mathbb C^3$, let $B_{\boldsymbol \beta}$ be the following differential operator on $E\times E$
\begin{equation*}
\begin{split}
B_{\boldsymbol \beta} =& \vert y-z\vert^2 \Delta_y\Delta_z\\
&+2(\beta_3+\beta_1+d) \sum_{j=1}^d (z_j-y_j) \frac{\partial}{\partial y_j} \Delta_z+2 (\beta_2+\beta_1+d) \sum_{j=1}^d (y_j-z_j) \frac{\partial}{\partial z_j} \Delta_y\\
&+(\beta_3+\beta_1+d)(\beta_3+\beta_1+2) \Delta_z+(\beta_2+\beta_1+d)(\beta_2+\beta_1+2) \Delta_y \\
&-2(\beta_3+\beta_1+d)(\beta_2+\beta_1+d) \sum_{j=1}^d \frac{\partial^2}{\partial y_j\partial z_j}\ .\\
\end{split}
\end{equation*}
Then
\begin{equation}
B_{\boldsymbol \beta}\, l_{\boldsymbol \beta+2_1} = b(\beta) l_{\boldsymbol \beta}
\end{equation}
where
\[b(\boldsymbol \beta) = (\beta_1+d)(\beta_1+2)(\beta_1+\beta_2+\beta_3+2d)(\beta_1+\beta_2+\beta_3+d+2)\ .
\]
\end{theorem}

\section{Applications of the Bernstein-Sato identity}

The first application of the Bernstein-Sato identity is the computation of the residues of the distribution $\mathcal L_{\boldsymbol \beta}$ along the plane $\mathcal H_k$ by induction over $k$.

\begin{proposition} Let $\boldsymbol \beta^0$ be such that 
$\beta_1^0+\beta_2^0+\beta_3^0=-2d-2k-2$ for some  $k\in \mathbb N$. Assume that $\beta_j\notin -d-2\mathbb N$ ($j=1,2,3$), and $\beta^0_1\neq -2$. Then, for $f\in \mathcal S(E)$ and $g\in \mathcal S(E\times E)$,
\begin{equation}
Res(\mathcal L_{\boldsymbol \beta}, \boldsymbol \beta^0) (f\otimes g\big) =\frac{1}{(2k+2)(2k+d)(\beta_1^0+2)(\beta_1^0+d)} Res(\mathcal L_{\boldsymbol \beta},{ \boldsymbol \beta}^0+2_1) (f\otimes B_{\boldsymbol \beta^0}^t \,g\big)
\end{equation}

\end{proposition}
\begin{proof}
For generic values of $\boldsymbol \beta$, from the Bernstein-Sato identitiy, 
\[\mathcal L_{\boldsymbol \beta} (f\otimes g) = (\ell_{\boldsymbol \beta}, f\otimes g) = \frac{1}{b(\boldsymbol \beta)}(B_{\boldsymbol \beta}\,\ell_{\beta+2_1}, f\otimes g)= \frac{1}{b(\boldsymbol \beta)}(\ell_{\boldsymbol \beta+2_1}, f\otimes B_{\boldsymbol \beta}^t\,g)\ ,
\]
and compute the residue at $\boldsymbol \beta^0$ on both sides. 
\end{proof}

Let $C_{\boldsymbol \beta}=B_{\boldsymbol \beta}^t$. Except for the change of parameters, this is nothing but the operator $F_{\lambda,\mu}$. 
\begin{proposition} 
\begin{equation}
\begin{split}
C_{\boldsymbol \beta} = B_{\boldsymbol \beta}^t  =& \ \ \vert y-z\vert^2 \,\Delta_y\,\Delta_z  \\
& +2(\beta_1+\beta_2+d+2) \sum_{j=1}^d (z_j-y_j) \frac{\partial}{\partial z_j}\Delta_y\\
& +2(\beta_1+\beta_3+d+2) \sum_{j=1}^d (y_j-z_j) \frac{\partial}{\partial y_j}\Delta_z\\\
&+(\beta_1+\beta_2+2d)(\beta_1+\beta_2+d+2) \Delta_y\\
&-2(\beta_1+\beta_2+d+2)(\beta_1+\beta_3+d+2)\sum_{j=1}^d \frac{\partial^2}{\partial y_j\partial z_j}\\
&+(\beta_1+\beta_3+2d)(\beta_1+\beta_3+d+2) \Delta_z\ .
\end{split}
\end{equation}
\end{proposition}
To write an expression for the residue at a pole in $\mathcal H_k $, where $k\in \mathbb N$, let use the following convention: for $\boldsymbol \beta = (\beta_1,\beta_2,\beta_3)$ and $k\in \mathbb N$, let  \[\boldsymbol \beta - (2k)_1 = (\beta_1-2k, \beta_2,\beta_3)\]

Now, for $\boldsymbol \beta \in \mathcal H_0$,  define the differential operator $E_{\boldsymbol \beta}^{(k)}$ 
on $E\times E$ by 
\begin{equation*}
C_{\boldsymbol \beta}^{(0)} = \Id ,\quad C_{\boldsymbol \beta}^{(k)} = C_{\boldsymbol \beta-2_1}\circ \dots \circ C_{\boldsymbol \beta - {(2k)}_1}
\end{equation*}

\begin{theorem} Let $\boldsymbol \beta^0\in \mathcal H_0$, and let $k\in \mathbb N$. Assume that $\beta_j\notin -d-2\mathbb N$ ($j=1,2,3)$ and $\beta^0_1\notin \{ 0,2,\dots, 2k-2\}$. Then

\begin{equation}
Res(\mathcal L_{\boldsymbol \beta}, \boldsymbol \beta^0 -(2k)_1) (f\otimes g) = c_k(\boldsymbol \beta^0) \int_E f(x) \big(C_{\boldsymbol \beta^0}^{(k)}g\big)(x,x) dx
\end{equation}
where
\[c_k(\boldsymbol \beta^0) = \frac{1}{16^k}\frac{1}{ k!}\frac{1}{(\rho)_k}\frac{1}{{(-\frac{\beta^0_1}{2})}_k}\frac{1}{{(-\frac{\beta^0_1}{2}-\rho+1)}_k}c_0(\boldsymbol \beta^0)\ .
\]
\end{theorem}

The Bernstein-Sato operator can also be used to describe a family of covariant bidifferential operators. Let $\lambda,\mu$ be in $\mathbb C$. For $k\in \mathbb N$, let $F_{\lambda,\mu}^{(k)}$ be the bidifferential operator defined by
\begin{equation}\label{covbidiff2}
F_{\lambda,\mu}^{(k)}f (x)= F_{\lambda+k-1,\,\mu+k-1}\circ\dots \circ F_{\lambda,\mu}f(x,x) ,
\end{equation}
where $F_{\lambda,\mu}$ is the differential operator on $E\times E$ defined by \eqref{flambdamu}.

\begin{theorem} Let $\lambda, \mu$ in $\mathbb C$, and $k\in \mathbb N$. Then the operator $F_{\lambda,\mu}^{(k)}$ is conformally covariant with respect to $(\pi_\lambda\otimes \pi_\mu, \pi_{\lambda+\mu+\rho+2k})$.
\end{theorem}
\begin{proof} Recall that the operator $F_{\lambda,\mu}$ is  covariant w.r.t. $(\pi_\lambda\otimes \pi_\mu, \pi_{\lambda+1}\otimes \pi_{\mu+1})$. So, by induction, 
$F_{\lambda+k-1,\, \mu+k-1}\circ\dots \circ F_{\lambda\mu}$ is covariant w.r.t. $(\pi_\lambda\otimes\pi_\mu, \pi_{\lambda+k}\otimes \pi_{\mu+k})$. Now the map
\[\mathcal C^\infty_c(E\times E)\ni f \longmapsto \overline {f} \in \mathcal C^\infty_c(E),
\]
where $\overline {f} (x) = f(x,x)$, is covariant w.r.t. $(\pi_{\lambda+k}\otimes \pi_{\mu+k}, \pi_{\lambda+\mu+\rho+2k})$. The assertion follows.
\end{proof}

For $k=1$, one gets
\[F_{\lambda,\,\mu}^{(1)} = 4(\mu+1)(\mu+\rho)\Delta_y -8 (\lambda+1)(\mu+1) R+4(\lambda+1)(\lambda+\rho)\Delta_z
\]
to be compared with \eqref{resh1}.

\section{Final remarks}

\medskip
\noindent
{\bf 1.} The construction of the covariant differential operator $N_{\lambda,\mu}$ admits a
natural generalization. Let $\tau$ be the standard representation of $G$ and $\tau'$
its dual representation.  Choose highest weight vectors $v$ and $\phi$ (with respect to some  suitable ordering) for $\tau$ and
$\tau'$, respectively. Then up to some constant, the multiplication by the matrix coefficient

\begin{equation*}
	g_1, g_2 \longmapsto \left< \tau(g_1)v,\tau'(g_2)\phi \right>
\end{equation*}
coincides with $M$ from lemma \ref{matrcoeff}. Upon replacing $\tau$ by another
irreducible finite-dimensional representation of $G$ one obtains a multiplication operator
that intertwines the tensor products of general (i.e. not necessarily spherical) principle
series representations (see \cite{b}). Using appropriate intertwining operators one obtains an
associated differential operator. 
We would also like to point out that the operators used by Oksak in his work \cite{ok} on invariant trilinear forms for 
$G = Sl_2({\mathbb{C}}) \simeq Spin(3,1)$ are of this type.  See also \cite{ka}.
\medskip

\noindent
{\bf 2.} A byproduct of Proposition \ref{res0} is that the residue at a point $\boldsymbol \beta \in \mathcal H_0$ vanishes identically if $\frac{-\beta_1-d}{2}\in -\mathbb N$, i.e. if $\beta_1\in -d+2m, m\in \mathbb N$. For such a value, observe that
\[\lambda_1 = \frac{\beta_2+\beta_3}{2}+\rho = \frac{-\beta_1-2d}{2} +\rho = -m\ .
\]
Now for $m\geq 1$, there exists a nontrivial differential operator $R_m$ (see \cite{c}) on $\mathcal C^\infty_c(E)$ which is covariant w.r.t. $(\pi_{-m}, \pi_m)$. Now let \[\tilde{\boldsymbol \beta}= (\beta_1-2m, \beta_2+2m, \beta_3+2m)\ ,\] so that  \[\widetilde {\boldsymbol \lambda}=(m, \lambda_2,\lambda_3)\ .\]
As $\widetilde {\beta_1}+\widetilde {\beta_2} +\widetilde{\beta_3} = -2d+2m$,  $\widetilde \beta$ is no longer a pole. So the form $\mathcal L_{\widetilde{\boldsymbol\beta}}$ is well defined, and the form
\[(f_1,f_2,f_3)\longmapsto \mathcal L_{\widetilde{\boldsymbol\beta}}(R_mf_1,f_2,f_3)
\]
is invariant with respect to $(\pi_{-m},\pi_{\lambda_2},\pi_{\lambda_3})$. However, the relation $\lambda_1+\lambda_2+\lambda_3 = -\rho$ guarantees that the form
\[(f_1,f_2,f_3)\longmapsto \int_E f_1(x)f_2(x)f_3(x) dx
\]
is invariant under $(\pi_{\lambda_1},\pi_{\lambda_2}, \pi_{\lambda_3})$. So, for $\lambda_1 = -m$ and $(\lambda_2,\lambda_3)$ generic, we have produced \emph{two} (linearly independant) trilinear invariant forms on $\mathcal C^\infty_c(E)\times\mathcal C^\infty_c(E)\times \mathcal C^\infty_c(E)$ w.r.t. $(\pi_{-m},\pi_{\lambda_2},\pi_{\lambda_3})$ . Although we won't develop these aspects here, the same remark can be used to produce, for specific values of $\boldsymbol \lambda$, two (linearly independant) bidifferential operators covariant under the same actions of the conformal group. Notice that this is in concordance with the results and the philosophy of  \cite{kr} and \cite {m}.

\medskip

\noindent

{\bf 3.} The relation between our formula for covariant bidifferential operators \eqref{covbidiff2} and the formul\ae   \ obtained in \cite{or} or in \cite{m} is still to be investigated, and the coefficients which relate them seem to be important. In the classical setting (i.e. for the original Rankin-Cohen operators acting on the upper half-plane), much effort has been devoted to understand the structure of this family of operators (see \cite{z}, \cite{cmz}, \cite{uu}, \cite{cm}, \cite{p}). We hope that our realization of these operators will add to the understanding of the family of generalized Rankin Cohen operators.

\bigskip
\footnotesize{ \noindent Addresses
\\(RB) Mathematisches Institut Universit\"at T\"ubingen, Auf der Morgenstelle 10,
72076 T\"ubingen,   Germany\\
(JLC) Institut \'Elie Cartan, Universit\'e Henri Poincar\'e (Nancy 1),
54506 Vandoeuvre-l\`es-Nancy, France.\\

\medskip

\noindent \texttt{{
 Ralf.Beckmann@uni-tuebingen.de, jlclerc@iecn.u-nancy.fr
}}

\end{document}